\documentclass[12pt, reqno]{amsart}
\usepackage{amssymb}
\usepackage{verbatim}
\usepackage[vcentermath]{youngtab}
\usepackage{float}
\usepackage{hyperref}
\usepackage{tikz}
\usepackage{hyperref}
\usepackage{multicol}
\usepackage{wrapfig}
\usepackage{multicol}
\usepackage[left=1.1in, right=1.1in, top=1in, bottom=.9in]{geometry}

\usepackage{times}
\usepackage[T1]{fontenc}
\usepackage{mathrsfs}
\usepackage{epsfig}
\usepackage{color}
\usepackage{array} 
\newcolumntype{L}{>{$}l<{$}}
\newcolumntype{R}{>{$}r<{$}}

\usepackage{enumitem}

\newtheorem{thm}{Theorem}[section]
\newtheorem{lemma}[thm]{Lemma}

\newtheorem{cor}[thm]{Corollary}
\newtheorem{prop}[thm]{Proposition}

\theoremstyle{remark}

\theoremstyle{definition}
\newtheorem{defn}[thm]{Definition}
\newtheorem{remark}[thm]{Remark}

\newtheorem{example}[thm]{Example}

\def\ZZ{\mathbb{Z}}
\def\QQ{\mathbb{Q}}
\def\PP{\mathbb{P}}
\def\NN{\mathbb{N}}
\def\OO{\mathcal{O}}

\def\FF{\mathbb{F}}

\def\RR{\mathbb{R}}
\def\CC{\mathbb{C}}
\def\GG{\mathbb{G}}

\def\U{\mathcal{U}}

\def\H{\mathcal{H}}

\def\wt{\widetilde}

\def\h{\widehat}

\def\U{\mathcal{U}}

\def\multi#1#2{\ensuremath{\left(\kern-.3em\left(\genfrac{}{}{0pt}{}{#1}{#2}\right)\kern-.3em\right)}}

\newcommand{\Res}{\mathrm{Res}}

\newcommand{\brmod}[1]{\,\,[\mathrm{mod}\,\, #1]}

\usepackage{empheq}

\numberwithin{equation}{section}

\begin{document}

\title[Universal relations for dynamical units]{Dynatomic polynomials, necklace operators, and\\ universal relations for dynamical units}

\author{John R. Doyle}
\address{Department of Mathematics\\
Oklahoma State University\\
Stillwater, OK 74078
}
\email{john.r.doyle@okstate.edu}

\author{Paul Fili}
\address{Department of Mathematics\\
Oklahoma State University\\
Stillwater, OK 74078
}
\email{paul.fili@okstate.edu}

\author{Trevor Hyde}
\address{Department of Mathematics\\
University of Chicago \\
Chicago, IL 60637\\
}
\email{tghyde@uchicago.edu}

\maketitle

\section{Introduction}
Let $f(x) \in K[x]$ be a polynomial with coefficients in a field $K$. For an integer $k \ge 0$, we denote by $f^k(x)$ the $k$-fold iterated composition of $f$ with itself.
The $d$th dynatomic polynomial $\Phi_{f,d}(x) \in K[x]$ of $f$ is defined by the product
\[
    \Phi_{f,d}(x) := \prod_{e\mid d}(f^{d/e}(x) - x)^{\mu(e)},
\]
where $\mu$ is the standard number-theoretic M\"obius function on $\NN$. We refer the reader to \cite[\S 4.1]{ADS} for background on dynatomic polynomials. For generic $f(x)$, the $d$th dynatomic polynomial $\Phi_{f,d}(x)$ vanishes at precisely the periodic points of $f$ with primitive period $d$. 
In this paper we consider the polynomial equation $\Phi_{f,d}(x) = 1$ and show that it often has $f$-preperiodic solutions determined by arithmetic properties of $d$, independent of $f$. Moreover, these $f$-preperiodic solutions are detected by cyclotomic factors of the \emph{$d$th necklace polynomial}:
\[
    M_d(x) = \frac{1}{d}\sum_{e\mid d}\mu(e)x^{d/e} \in \QQ[x].
\]
Our results extend earlier work of Morton and Silverman \cite{MS}, but our techniques are quite different and apply more broadly.

We begin by recalling some notation and terminology. The \emph{cocore} of a positive integer $d$ is $d/d'$ where $d'$ is the largest squarefree factor of $d$. If $m\geq 0$ and $n\geq 1$, then the \emph{$(m,n)$th generalized dynatomic polynomial} $\Phi_{f,m,n}(x)$ of $f(x)$ is defined by $\Phi_{f,0,n}(x) := \Phi_{f,n}(x)$ and
\[
    \Phi_{f,m,n}(x) := \frac{\Phi_{f,n}(f^m(x))}{\Phi_{f,n}(f^{m-1}(x))}
\]
for $m \geq 1$. The roots of $\Phi_{f,m,n}$ for generic $f$ are those preperiodic points which enter into an $n$-cycle after exactly $m$ iterations under $f$.

Theorem \ref{thm intro main} is our main result; it is proved in Section \ref{sec results}.

\begin{thm}
\label{thm intro main}
Let $K$ be a field, let $f(x) \in K[x]$ be a polynomial of degree at least 2, and let $c, d, m, n$ be integers with $c, m\geq 0$ and $d, n \geq 1$. Suppose that
\begin{enumerate}
    \item\label{item:divisible} either $m > c$ or $n\nmid d$,
    \item\label{item:thm1-2} the cocore of $d$ is at least $m - \max(c-1, 0)$, and
    \item\label{item:thm1-3} $x^n - 1$ divides the $d$th necklace polynomial $M_d(x)$ in $\QQ[x]$.
\end{enumerate}
Then $\Phi_{f,m,n}(x)$ divides $\Phi_{f,c,d}(x) - 1$.

Alternatively, if $d > 1$, $c - 1 \geq m$, and $n = 1$, then $\Phi_{f,m,n}(x)$ divides $\Phi_{f,c,d}(x) - 1$.
\end{thm}

\begin{remark}
While we generally discuss polynomials over arbitrary fields, the polynomial $M_d$ will always be considered to be a polynomial over $\QQ$; in particular, all statements regarding divisibility or factorizations of $M_d$ should be interpreted in characteristic zero.
\end{remark}

\subsection{Dynamical units}

Let $K$ be a number field with ring of integers $\OO_K$. Morton and Silverman \cite{MS} define \emph{dynamical units} to be algebraic integral units constructed in one of several closely related ways from differences of preperiodic points of a given monic polynomial $f(x) \in \OO_K[x]$. If $\Phi_{f,m,n}(x)$ divides $\Phi_{f,c,d}(x) - 1$, then for each root $\alpha \in \overline{K}$ of $\Phi_{f,m,n}(x)$,
\begin{equation}
\label{eqn unit relation}
    1 = \Phi_{f,d}(\alpha) = \prod_{\beta}(\alpha - \beta),
\end{equation}
where the product ranges over all the roots $\beta$ of $\Phi_{f,c,d}(x)$ with multiplicity. The differences $\alpha - \beta$ are dynamical units and \eqref{eqn unit relation} is a multiplicative relation between dynamical units. If the conditions of Theorem \ref{thm intro main} are satisfied for $m, n, c, d$, then \eqref{eqn unit relation} holds for all $f(x)$ with degree at least 2; we view these as \emph{universal relations} for dynamical units.
Examples of universal relations for dynamical units have been found by Morton and Silverman \cite[Thm. 7.5]{MS} and Benedetto \cite[Thm. 2]{benedetto}.
We give some results on universal relations, and relate them to previous work, in Section \ref{sec:dynamical units} below.

\subsection{Cyclotomic factors of necklace polynomials}
Of the conditions in Theorem \ref{thm intro main}, (3) is the most subtle. Necklace polynomials $M_d(x)$ have several combinatorial interpretations; for example, if $q$ is a prime power, then $M_d(q)$ is the number of irreducible degree-$d$ monic polynomials in $\FF_q[x]$. These interpretations give no indication as to when, if ever, $M_d(x)$ will vanish at all the $n$th roots of unity. However, as observed in \cite{hyde_cyclo}, necklace polynomials are generally divisible by many cyclotomic polynomials. Recall that the \emph{$n$th cyclotomic polynomial} $\Phi_n(x)$ is the $\QQ$-minimal polynomial of a primitive $n$th root of unity. 

\begin{example}
$M_{105}(x)$ factors over $\QQ$ as
\begin{align}
\label{eqn neck factor}
    M_{105}(x) &= \tfrac{1}{105}(x^{105} - x^{35} - x^{21} - x^{15} + x^7 + x^5 + x^3 - x)\nonumber\\
    &= e(x)\cdot \Phi_8\cdot \Phi_6\cdot \Phi_4\cdot \Phi_3\cdot \Phi_2\cdot \Phi_1\cdot x,
\end{align}
where $e(x) \in \QQ[x]$ is a degree 92, irreducible, non-cyclotomic polynomial. Since 
\[
    x^n - 1 = \prod_{m\mid n} \Phi_m(x),
\]
the factorization \eqref{eqn neck factor} implies that $M_{105}(x)$ is divisible by $x^n - 1$ for $n = 1, 2, 3, 4, 6, 8$. Note that $d = 105 = 3 \cdot 5 \cdot 7$ is squarefree, hence the cocore of $d$ is 1. Thus Theorem \ref{thm intro main} implies that for any polynomial $f(x) \in K[x]$ of degree at least 2, $\Phi_{f,105}(x) - 1$ is divisible by $\Phi_{f,1,n}(x)$ for $n = 1, 2, 3, 4, 6, 8$ and $\Phi_{f,0,n}$ for $n = 2, 4, 6, 8$.
\end{example}

In light of Theorem \ref{thm intro main} one might naturally ask how often $x^n - 1$ divides $M_d(x)$. Figure \ref{fig:neck data plot} suggests that $M_d(x)$ is divisible by several $x^n - 1$ for all $d \geq 1$.
\begin{figure}[h]
    \centering
    \includegraphics[scale = .5]{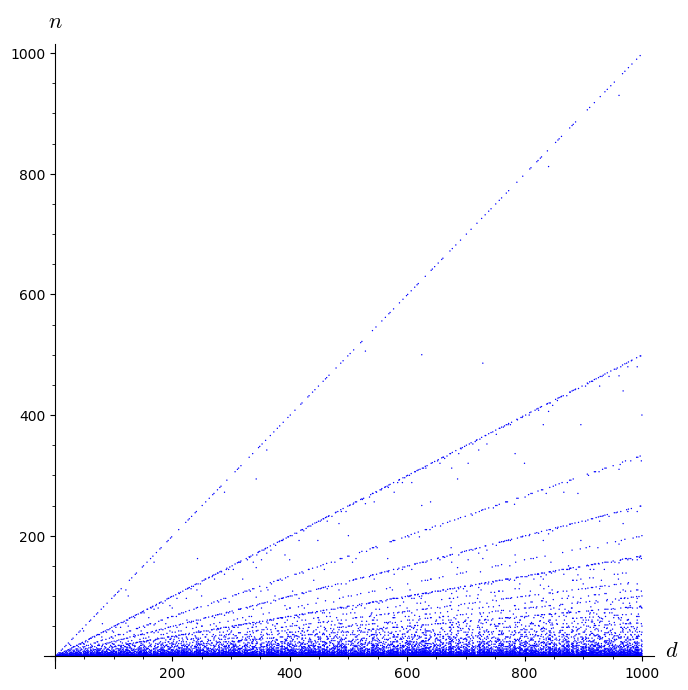}
    \caption{Pairs $(d,n)$ with $d, n \leq 1000$ for which $x^n - 1$ divides $M_d(x)$.}
    \label{fig:neck data plot}
\end{figure}
Hyde \cite{hyde_cyclo} characterized the cyclotomic factors of necklace polynomials in terms of hyperplane arrangements in finite abelian groups. Let $\h{\U}_n := \mathrm{Hom}((\ZZ/(n))^\times, \CC^\times)$ denote the group of Dirichlet characters of modulus $n$. If $q$ is a unit modulo $n$, then the \emph{hyperplane} $\H_q \subseteq \h{\U}_n$ is defined to be the set
\[
    \H_q := \{\chi \in \h{\U}_n : \chi(q) = 1\}.
\]

The following theorem gives an alternative to condition (3) in Theorem \ref{thm intro main} in terms of hyperplanes in the group of Dirichlet characters. We prove Theorem \ref{thm intro hyperplane} in Section \ref{sec cyclo factors of neck}.

\begin{thm}
\label{thm intro hyperplane}
Let $d, n \geq 1$. Then $x^n - 1$ divides $M_d(x)$ if and only if
\[
    \h{\U}_n \subseteq \bigcup_{\substack{p\mid d\\p\nmid n}} \H_p.
\]
\end{thm}

Theorem \ref{thm intro hyperplane} says $x^n - 1$ divides $M_d(x)$ if and only if the finite abelian group $\h{\U}_n$ of modulus $n$ Dirichlet characters is covered by an arrangement of ``hyperplanes'' determined by the prime factors of $d$. In Example \ref{ex m = 65} we explain how the 5 distinct prime factors of
\[
    d = 440512358437 = 47^2 \cdot 73 \cdot 79 \cdot 151 \cdot 229
\]
correspond to the 5 lines in $(\RR/4\ZZ)^2$ in Figure \ref{fig:65 intro}, and how the fact that the lines cover all the lattice points translates, via Theorem \ref{thm intro hyperplane}, into the fact that $x^{65} - 1$ divides $M_{440512358437}(x)$. Since the cocore of $d$ is 47, Theorem \ref{thm intro main} implies that
\[
    \Phi_{f,m,65}(x) \text{ divides } \Phi_{f,440512358437}(x) - 1,
\]
for all $f(x) \in K[x]$ with $\deg(f) \geq 2$ and $0 \leq m \leq 47$.
\begin{figure}[h]
    \centering
    \includegraphics[scale = .14]{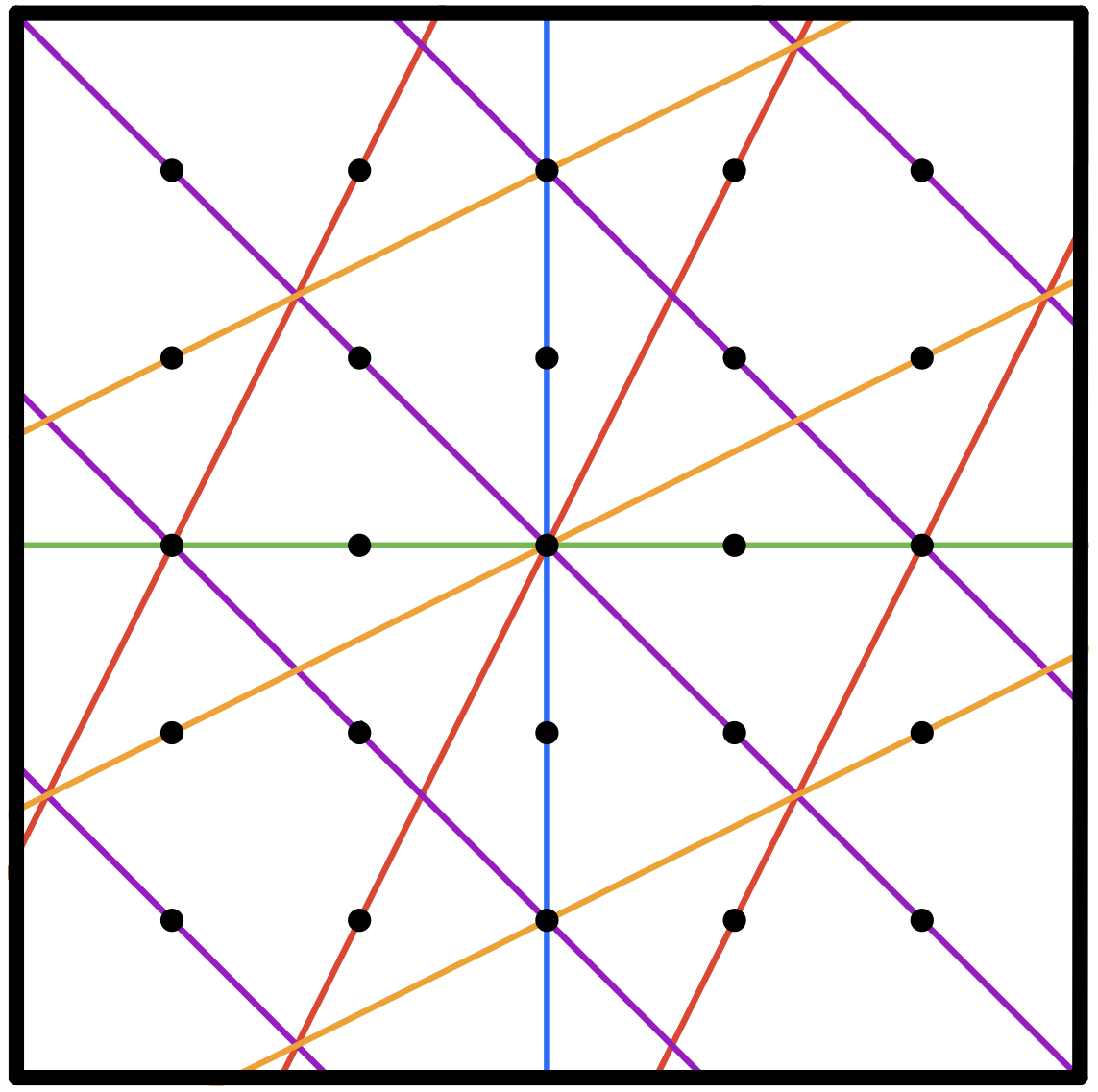}
    \caption{}
    \label{fig:65 intro}
\end{figure}

\subsection{Cyclotomic factors of shifted cyclotomic polynomials}
Cyclotomic factors of necklace polynomials are also closely related to cyclotomic factors of shifted cyclotomic polynomials $\Phi_d(x) - 1$. For example, if $d = 105$, then
\[
    \Phi_{105}(x) - 1 = \wt{e}(x) \cdot \Phi_8 \cdot \Phi_6 \cdot \Phi_4 \cdot \Phi_3 \cdot \Phi_2 \cdot \Phi_1\cdot x,
\]
where $\wt{e}(x) \in \QQ[x]$ is a degree 35, irreducible, non-cyclotomic polynomial. Note that the cyclotomic factors dividing $\Phi_{105}(x) - 1$ are precisely the same as those dividing $M_{105}(x)$. 
In general, $M_d(x)$ and $\Phi_d(x) - 1$ have most, but not all, cyclotomic factors in common.
See \cite{hyde_cyclo} for a detailed analysis of the cyclotomic factors in these two sequences.

Cyclotomic factors of $\Phi_d(x) - 1$ are also detected by cyclotomic factors of $M_d(x)$ and have an interpretation in terms of multiplicative relations between \emph{cyclotomic units} analogous to the situation with dynamical units discussed above. Thus the cyclotomic factors of necklace polynomials give explicit structural parallels between these two analogous families of units.

\subsection{Necklace operators}
Let $\ZZ\Psi$ denote the ring generated by formal expressions $[m]$ with $m \in \NN$ subject only to the multiplicative relations $[m][n] = [mn]$. The \emph{$d$th necklace operator} $\varphi_d \in \ZZ\Psi$ is defined by
\[
    \varphi_d := \sum_{e\mid d}\mu(e)[d/e].
\]
The cyclotomic factors of $M_d(x)$, cyclotomic factors of $\Phi_d(x) - 1$, and dynatomic factors of $\Phi_{f,d}(x) - 1$ ultimately trace back to the necklace operator $\varphi_d$. 
The polynomials $M_d(x)$, $\Phi_d(x)$ and $\Phi_{f,d}(x)$ may be expressed as images of $\varphi_d$ with respect to different $\ZZ\Psi$-module structures. Suppressing the details of the module structures for now, we have
\begin{align}
\label{eqn dynatomic intuitive}
    M_d(x) &= \frac{1}{d} \sum_{e\mid d}\mu(e)x^{d/e} = \varphi_d (x/d),\nonumber\\
    \Phi_d(x) &= \prod_{e\mid d}(x^{d/e} - 1)^{\mu(e)} = (x - 1)^{\varphi_d},\nonumber\\
    \Phi_{f,d}(x) &= \prod_{e\mid d}(f^{d/e}(x) - x)^{\mu(e)} = (f(x) - x)^{\varphi_d}.
\end{align}

As the notation suggests, $M_d(x)$ is an image in an additive $\ZZ\Psi$-module while $\Phi_d(x)$ and $\Phi_{f,d}(x)$ arise from multiplicative $\ZZ\Psi$-modules. Much of the work that goes into proving Theorem \ref{thm intro main} involves constructing the appropriate $\ZZ\Psi$-module in which to realize the above expression of $\Phi_{f,d}(x)$ as an image of $\varphi_d$.

All of the cyclotomic and dynatomic factors of the polynomials discussed above, as well as the connection to hyperplane arrangements in the group of Dirichlet characters, ultimately traces back to the following factorization of the necklace operator (in a localization of $\ZZ\Psi$):
\[
    \varphi_d = [d] \prod_{p \mid d}\Big(1 - \frac{1}{[p]}\Big),
\]
where the product is taken over all primes $p$ dividing $d$.

\subsection{Acknowledgements}
We are happy to thank Patrick Morton and Joe Silverman for feedback on an earlier draft. 
John Doyle was partially supported by NSF grant DMS-2112697.
Trevor Hyde was partially supported by the NSF Postdoctoral Research Fellowship DMS-2002176 and the Jump Trading Mathlab Research Fund.

\section{Preliminary results}

In this section we prove preliminary results leading up to the proofs of Theorem \ref{thm intro main} and Theorem \ref{thm intro hyperplane} in Section \ref{sec results}. Our main goal is to make sense of \eqref{eqn dynatomic intuitive}. We accomplish this by introducing the notions of composition rings and their algebras.
In Section \ref{sec generically sqfree}, we prove a statement on the generic separability of (generalized) dynatomic polynomials; this result is folklore in the arithmetic dynamics community but we were unable to find a suitable reference.

\subsection{Composition rings}
Suppose $R$ is a commutative ring and $S$ is a monoid of ring endomorphisms of $R$ with respect to composition. The monoid $S$ generates a (non-unital) subring $C_S$ of the ring of all $R$-valued functions on $R$, with pointwise ring operations. Furthermore, $C_S$ has an extra layer of structure coming from the composition operation on $S$. We abstract this situation into the notion of a composition ring.

\begin{defn}
\label{def comp ring}
A \emph{composition ring} $C$ is a (potentially non-unital) commutative ring together with an associative operation $\circ$ such that for all $f, g, h \in C$
\begin{enumerate}
    \item $(f + g)\circ h = (f \circ h) + (g \circ h)$,
    \item $(f \cdot g) \circ h = (f\circ h)\cdot (g\circ h)$, and
    \item there exists a two-sided compositional identity $x \in C$.
\end{enumerate}
A morphism $\sigma: C \rightarrow D$ of composition rings is a ring homomorphism which respects the composition operator and preserves compositional identities.
\end{defn}

All of the composition rings we consider are constructed as follows.

\begin{defn}
\label{def free comp}
Let $S$ be a multiplicative monoid. The \emph{free $S$-composition ring} $\ZZ\{S\}$ is the composition ring generated by expressions $[s]$ with $s \in S$ where the composition operation $\circ$ is determined by the following relations: for all $f, g \in \ZZ\{S\}$ and $s, t \in S$
\begin{enumerate}
    \item[(\textit{i})] $[s] \circ (f + g) = ([s] \circ f) + ([s] \circ g)$,
    \item[(\textit{ii})] $[s] \circ (f\cdot g) = ([s] \circ f)\cdot ([s]\circ g)$, and
    \item[(\textit{iii})] $[s] \circ [t] = [st]$.
\end{enumerate}
Note that the compositional identity is $x := [1]$ where $1 \in S$ is the multiplicative identity.
\end{defn}

To see that the composition operation on $\ZZ\{S\}$ is determined by these properties, first observe that Definition \ref{def comp ring} (1) and (2) reduce the computation of $f \circ g$ for $f, g \in \ZZ\{S\}$ to $[s] \circ g$ with $s \in S$. Then Definition \ref{def free comp} (\textit{i}) and (\textit{ii}) reduce us further to $[s] \circ [t]$ for $s, t \in S$, and finally (\textit{iii}) tells us that $[s] \circ [t] = [st]$. This reduction is illustrated in the following example.

\begin{example}
Let $S := \langle f, g\rangle$ be the free monoid on two generators. Consider the elements
\begin{align*}
    \alpha &:= 3[f^2][f] + 2[1][g]\\
    \beta &:= [f][g] + [fg]
\end{align*}
of $\ZZ\{S\}$. Then by Definition \ref{def comp ring} (1) and (2),
\begin{align*}
    \alpha \circ \beta &= (3[f^2][f] + 2[1][g])\circ \beta\\
    &= 3([f^2]\circ \beta)([f]\circ \beta) + 2([1]\circ \beta)([g]\circ \beta).
\end{align*}
Definition \ref{def free comp} (\textit{i}) and (\textit{ii}) imply that
\begin{align*}
    [f^2]\circ \beta &= [f^3][f^2g] + [f^3g],\\
    [f]\circ \beta   &= [f^2][fg] + [f^2g],\\
    [1]\circ \beta   &= [f][g] + [fg],\\
    [g]\circ \beta   &= [gf][g^2] + [gfg].
\end{align*}
Thus,
\[
    \alpha \circ \beta = 3\left([f^3][f^2g] + [f^3g]\right)\left([f^2][fg] + [f^2g]\right) + 2\left([f][g] + [fg]\right)\left([gf][g^2] + [gfg]\right).
\]
\end{example}

\begin{remark}
The composition ring $\ZZ\{S\}$ is closely related to the more familiar monoid ring $\ZZ[S]$. The latter is the ring generated by $[s]$ for $s \in S$ with multiplication determined by $[s]\cdot [t] = [st]$. The monoid ring $\ZZ[S]$ embeds into $\ZZ\{S\}$ as linear combinations of the ``degree one'' elements with product structure given by $\circ$.
\end{remark}

If $\wt{\sigma}: S \rightarrow T$ is a monoid homomorphism, then there is a unique composition ring homomorphism $\sigma: \ZZ\{S\} \rightarrow \ZZ\{T\}$ which lifts $\wt{\sigma}$. In fact, the map $S \mapsto \ZZ\{S\}$ gives a functor from monoids to composition rings. 

We further restrict our attention to monoids $S$ which are quotients of the free cyclic monoid on one generator $\langle f \rangle$. For each $m, n \in \NN$ with $n\geq 1$, let $\ZZ\{f\} := \ZZ\{\langle f \rangle\}$ and let 
\[
    \ZZ_{m,n}\{f\} := \ZZ\{\langle f : f^{m+n} = f^m\rangle\}.
\]
The monoid quotient
\[
    \langle f\rangle \rightarrow \langle f : f^{m+n} = f^m\rangle.
\]
induces, by functoriality, a map of composition rings $\ZZ\{f\} \rightarrow \ZZ_{m,n}\{f\}$. If $\alpha, \beta \in \ZZ\{f\}$ are elements with the same image in $\ZZ_{m,n}\{f\}$, then we write
\[
    \alpha \equiv \beta \bmod \ZZ_{m,n}\{f\}.
\]

\subsection{$\Psi$-module structure on $\ZZ\{f\}$}

Let $\NN^\circ$ denote the multiplicative monoid of natural numbers and let $\Psi := \NN[\NN^\circ]$ denote the monoid semiring of $\NN^\circ$.
That is, $\Psi$ is the semiring additively spanned by formal expressions $[m]$ for $m \in \NN$ such that for $m, n \in \NN$,
\[
    [m][n] = [mn].
\]

For each $m \in \NN$ there is a unique endomorphism $[m]$ of the cyclic semigroup $\langle f \rangle$ expressed in exponential notation as $f^{[m]} := f^m$. This gives, by functoriality, an endomorphism $[m] : \ZZ\{f\} \rightarrow \ZZ\{f\}$ of composition rings. We extend this action to a multiplicative $\Psi$-module structure on $\ZZ\{f\}$. 
\begin{example}
If $\psi = 3[5] + 2[4] \in \Psi$, then
\[
    ([f] - [1])^\psi = ([f] - [1])^{3[5] + 2[4]} := ([f^5] - [1])^3([f^4] - [1])^2.
\]
\end{example}

If $m\geq 0$ and $n \geq 1$ are natural numbers, the semiring quotient $\NN \rightarrow \NN/(m + n = m)$ induces a quotient on multiplicative monoids $\NN^\circ \rightarrow (\NN/(m + n = n))^\circ$. Let $\Psi_{m,n}$ denote the semiring quotient of $\Psi$ induced by this quotient of monoids. If $\psi_1, \psi_2 \in \Psi$ are two elements with the same image under this map, then we write
\[
    \psi_1 \equiv \psi_2 \brmod{m + n = m},
\]
or simply
\[
    \psi_1 \equiv \psi_2 \brmod{n},
\]
when $m = 0$. This notation is meant to suggest that the quotient takes place inside the brackets.

\begin{example}
If $m = 0$ and $n = 3$, then
\[
    5[1] - 3[2] + 4[5] \equiv 5[1] + [2] \not\equiv 2[1] + [2] \brmod{3}.
\]
The first congruence holds because $[2] \equiv [5] \brmod{3}$. The second congruence does not hold because the congruence does not extend to the coefficients so that $2[1] \not\equiv 5[1] \brmod{3}$.
\end{example}

The action of $\NN^\circ$ on the cyclic monoid $\langle f : f^{m+n} = f^m\rangle$ factors through the quotient $\NN^\circ/(m + n = m)$, hence the multiplicative $\Psi$-module structure on $\ZZ_{m,n}\{f\}$ factors through $\Psi_{m,n}$. Lemma \ref{lem compatible quotients} formally states this observation.

\begin{lemma}
\label{lem compatible quotients}
If $\alpha \in \ZZ\{f\}$ and $\psi_1, \psi_2 \in \Psi$ are such that $\psi_1 \equiv \psi_2 \brmod{m + n = m}$, then $\alpha^{\psi_1} \equiv \alpha^{\psi_2} \bmod \ZZ_{m,n}\{f\}$.
\end{lemma}

\subsection{Necklace operators}

If $R$ is a semiring, then let $R\Psi := R \otimes_\NN \Psi$ denote the extension of scalars of $\Psi$ from $\NN$ to $R$.

\begin{defn}
If $d\geq 1$ is a natural number, then the \emph{$d$th necklace operator} $\varphi_d$ is
\[
    \varphi_d := \sum_{e\mid d} \mu(e) [d/e] \in \ZZ\Psi,
\]
where $\mu$ is the usual number theoretic M\"obius function.
\end{defn}

There is a unique cancellation-free way to write the $d$th necklace operator as a difference $\varphi_d = \varphi_d^+ - \varphi_d^-$ of elements $\varphi_d^\pm \in \Psi$. Now let $\Phi_{f,d}^\pm \in \ZZ\{f\}$ be defined by
\[
    \Phi_{f,d}^\pm = ([f] - [1])^{\varphi_d^\pm}.
\]

Note that $\Psi$ and $\Psi_{m,n}$ have no additive torsion, hence embed into $\QQ\Psi$ and $\QQ\Psi_{m,n}$, respectively. Lemma \ref{lem free module} constructs a simple polynomial model of the free $\QQ\Psi_{m,n}$-module which allows us to relate the vanishing of $\varphi_d$ in $\ZZ\Psi_{m,n}$ to cyclotomic factors of $M_d(x)$. The polynomial ring $\QQ[x]$ carries a natural $\QQ\Psi$-module structure determined by $[k]g(x) := g(x^k)$ for $g(x) \in \QQ[x]$. Here $x^k$ denotes a monomial and not the $k$th compositional power of the identity function (which would again be the identity.)

\begin{lemma}
\label{lem free module}
Let $m\geq 0$ and $n\geq 1$. The $\QQ\Psi$-module structure on $\QQ[x]$ defined by $[k]g(x) := g(x^k)$ descends to $\QQ[x]/(x^{m+n} - x^m)$ and factors through $\QQ\Psi_{m,n}$. Furthermore, 
\[
    \QQ[x]/(x^{m+n} - x^m) \cong \QQ\Psi_{m,n}
\]
as $\QQ\Psi_{m,n}$-modules.
\end{lemma}

\begin{proof}
Let $M_{m,n} := \QQ[x]/(x^{m+n} - x^m)$. To see that the $\QQ\Psi$-module structure on $\QQ[x]$ descends to $M_{m,n}$ it suffices to check that if $f(x) \equiv g(x) \bmod (x^{m+n} - x^m)$, then $f(x^k) \equiv g(x^k) \bmod (x^{m+n} - x^m)$. This follows from the observation that $x^{mk}(x^{nk} - 1)$ is divisible by $x^m(x^n - 1)$ for all $k \in \NN$. The $\QQ\Psi$-action on $M_{m,n}$ clearly factors through $\QQ\Psi_{m,n}$. Observe that $M_{m,n}$ is cyclic as a $\QQ\Psi_{m,n}$-module and is generated by $x$. Note that both $M_{m,n}$ and $\QQ\Psi_{m,n}$ have dimension $m + n$ over $\QQ$, hence $M_{m,n}$ is free.
\end{proof}

\begin{defn}
The \emph{core} of a positive integer $d$ is the largest squarefree factor $d'$ of $d$ and the \emph{cocore} of $d$ is $d/d'$. Note that the core of $d$ is the product of all distinct primes dividing $d$.
\end{defn}

\begin{defn}
The \emph{$d$th necklace polynomial} $M_d(x) \in \QQ[x]$ for $d\geq 1$ is defined by
\[
    M_d(x) := \frac{1}{d}\sum_{e\mid d}\mu(e)x^{d/e}.
\]
\end{defn}

\begin{prop}
\label{prop necklace vanishing}
Let $m, n, d \in \NN$ be such that $n, d \geq 1$. If
\begin{enumerate}
    \item the cocore of $d$ is at least $m$, and
    \item $x^n - 1$ divides $M_d(x)$ in $\QQ[x]$,
\end{enumerate}
then $\varphi_d = 0$ in $\ZZ\Psi_{m,n}$ and 
\[
    \Phi_{f,d}^+ \equiv \Phi_{f,d}^- \bmod \ZZ_{m,n}\{f\}.
\]
\end{prop}

\begin{proof}
Lemma \ref{lem free module} implies that $\QQ[x]/(x^{m+n} - x^m)$ is a free $\QQ\Psi_{m,n}$-module generated by $x$. Hence $\varphi_d = 0$ in $\ZZ\Psi_{m,n}$ if and only if $\varphi_d x = 0$ in $\QQ[x]/(x^m(x^n - 1))$. Since
\[
    \varphi_d x = \sum_{e\mid d}\mu(e)[d/e]x = \sum_{e\mid d}\mu(e)x^{d/e} = dM_d(x),
\]
$\varphi_d = 0$ in $\ZZ\Psi_{m,n}$ if and only if $x^m$ and $x^n - 1$ both divide $M_d(x)$. Since $\mu(e) = 0$ when $e$ is not squarefree, the exponent of the largest power of $x$ dividing $M_d(x)$ is the cocore of $d$. Therefore (1) and (2) imply that $\varphi_d = 0$ in $\ZZ\Psi_{m,n}$.

If $\varphi_d = 0$ in $\ZZ\Psi_{m,n}$, then $\varphi_d^+ \equiv \varphi_d^- \brmod{m +n = m}$ and, by Lemma \ref{lem compatible quotients},
\[
    \Phi_{f,d}^+ = ([f] - [1])^{\varphi_d^+} \equiv ([f] - [1])^{\varphi_d^-} = \Phi_{f,d}^- \bmod \ZZ_{m,n}\{f\}.\qedhere
\]
\end{proof}

\subsection{Dynatomic polynomials are generically squarefree}
\label{sec generically sqfree}
We step aside from the theory developed in the previous sections to prove a dynamical lemma.

\begin{lemma}
\label{lem dynatomic squarefree}
Let $K$ be a field and let $f(x)$ be the generic degree $k\geq 2$ polynomial over $K$,
\[
    f(x) = a_kx^k + a_{k-1}x^{k-1} + \ldots + a_1x + a_0 \in K(a_0, a_1, \ldots, a_k)[x].
\]
Then for any $m, n \in \NN$ such that $n \geq 1$, $f^{m+n}(x) - f^m(x)$ has non-vanishing discriminant.
\end{lemma}

\begin{proof}
It suffices to prove the claim after specializing some subset of the coefficients of $f$. We consider two specializations depending the characteristic $p \geq 0$ of $K$.

First, suppose that $p \nmid k$. Morton \cite[Lemma 2]{morton:1996} shows that for $f_t(x) := x^k + t$, the polynomial $f_t^n(x) - x$ is separable over $K(t)$ for all $n \ge 1$, and, using similar techniques, the same is shown in \cite[Lemma 4.2]{doyle/poonen} for $f_t^{m+n}(x) - f_t^m(x)$ with $m \ge 0$ and $n \ge 1$.

Now suppose that $p \mid k$, and consider the polynomial $f_t(x) := x^k + tx \in K(t)[x]$. Then $f_t'(x) = t$, hence $(f_t^\ell)'(x) = t^\ell$ for all $\ell \ge 1$ by the chain rule. This implies that the polynomial $f_t^{m+n}(x) - f_t^m(x)$ has derivative $t^{m+n} - t^m$, a nonzero constant in $K(t)$. Since its derivative is nowhere vanishing, the polynomial $f_t^{m+n}(x) -f_t^m(x)$ is separable for all $m \ge 0$ and $n \ge 1$.
\end{proof}

\begin{remark}
\mbox{}
\begin{enumerate}
\item In characteristic $0$, Lemma \ref{lem dynatomic squarefree} predates \cite{morton:1996}. Indeed, for $a \in \CC$ and $f_a(x) = x^k + a$, the polynomial $f_a^{m+n}(x) - f_a^m(x)$ has a multiple root if and only if either $f_a$ has fewer than $k^n$ points of period dividing $n$, or $m \ge 1$ and the critical point $0$ is a root of $f_a^{m+n}(x) - f_a^m(x)$. The set of such $a \in \CC$ is contained in the degree-$k$ ``Multibrot set'' $\mathcal M_k$, which is a compact subset of $\CC$, hence one can further specialize $f_t$ to any $a \in \CC \setminus \mathcal M_k$. See also \cite{fakhruddin}---especially \cite[\textsection 3]{fakhruddin}---for related results.
\item In \cite[Lemma 4.2]{doyle/poonen}, which we refer to in the proof of Lemma~\ref{lem dynatomic squarefree}, it was assumed that $K$ is a finite field, since that was the only case for which the result was needed. However, the proof that $f_t^{m+n}(x) - f_t^m(x)$ is separable over $K(t)$ only requires that the characteristic of $K$ does not divide $k$.
\end{enumerate}
\end{remark}

\begin{defn}
If $f(x) \in K[x]$ is a polynomial, then the \emph{$n$th dynatomic polynomial} $\Phi_{f,n}(x) \in K[x]$ for $n\geq 1$ is defined by the product
\[
    \Phi_{f,n}(x) := \prod_{j\mid n}(f^{n/j}(x) - x)^{\mu(j)}.
\]
If $m\geq 0$, then the \emph{$(m,n)$th generalized dynatomic polynomial} $\Phi_{f,m,n}(x)$ is defined by $\Phi_{f,0,n}(x) := \Phi_{f,n}(x)$ and for $m\geq 1$,
\[
    \Phi_{f,m,n}(x) := \frac{\Phi_{f,n}(f^m(x))}{\Phi_{f,n}(f^{m-1}(x))}.
\]
\end{defn}

Despite their appearance, dynatomic polynomials are indeed polynomials and not just rational functions, as was first proven by Morton and Patel \cite{morton patel}.
See Silverman \cite[Sec. 4.1]{ADS} for a general introduction to dynatomic polynomials and \cite[Thm. 4.5]{ADS} for a proof that $\Phi_{f,d}(x)$ is a polynomial (and not just a rational function as is apparent from the defining product). As a special case of Hutz \cite[Thm. 1]{hutz} we get that $\Phi_{f,m,n}(x)$ is a polynomial; we may also deduce this quickly from Lemma \ref{lem dynatomic squarefree}.

The following factorization of $f^{m+n}(x) - f^m(x)$ is well-known and is often used without proof. We prove it here for the reader's convenience.

\begin{lemma}
\label{lem dynamic factor}
Let $f(x) \in K[x]$ be a polynomial of degree at least 2, then
\[
    f^{m+n}(x) - f^m(x) = \prod_{\substack{i \leq m\\ j \mid n}} \Phi_{f,i,j}(x).
\]
\end{lemma}

\begin{proof}
Recall that the definition of the dynatomic polynomials is equivalent to
\begin{equation}
\label{eqn dynamic factor}
    f^n(x) - x = \prod_{j\mid n} \Phi_{f,j}(x),
\end{equation}
by M\"obius inversion. Pre-composing both sides with $f^m(x)$ and using the telescoping product identity
\[
    \Phi_{f,j}(f^m(x)) = \frac{\Phi_{f,j}(f^m(x))}{\Phi_{f,j}(f^{m-1}(x))}\frac{\Phi_{f,j}(f^{m-1}(x))}{\Phi_{f,j}(f^{m-2}(x))}\cdots \frac{\Phi_{f,j}(f(x))}{\Phi_{f,j}(x)}\Phi_{f,j}(x)
    = \prod_{i\leq m} \Phi_{f,i,j}(x),
\]
gives us the desired factorization of $f^{m+n}(x) - f^m(x)$.
\end{proof}

Together Lemma \ref{lem dynatomic squarefree} and Lemma \ref{lem dynamic factor} imply that the generic generalized dynatomic polynomial $\Phi_{f,m,n}(x)$ is also squarefree.

\subsection{Composition algebras}
Next we introduce the notion of an algebra for a composition ring.

\begin{defn}
Let $C$ be a composition ring. A \emph{$C$-composition algebra} is a commutative ring $R$ together with an operation $\circ: R \times C \rightarrow R$ such that for all $r \in R$ and $g, h \in C$ we have
\begin{enumerate}
    \item $r \circ (g \circ h) = (r \circ g) \circ h$,
    \item $r\circ (g + h) = (r \circ g) + (r \circ h)$,
    \item $r\circ (g\cdot h) = (r\circ g)\cdot (r\circ h)$, and
    \item $r \circ x = r$,
\end{enumerate}
where $x$ is the compositional identity in $C$.
\end{defn}

Suppose that a monoid $S$ acts (on the right) by ring endomorphisms on a commutative ring $R$. If $r \in R$ and $s \in S$, then we denote this action by $r^s$. By construction there is a unique way to extend this action to a $\ZZ\{S\}$-composition algebra structure on $R$ so that
\[
    r \circ [s] = r^s
\]
for all $r \in R$ and $s \in S$.

Let $K$ be a field. The polynomial ring $K[x]$ is the free $K$-algebra on one generator. This implies that for any element $f$ in a $K$-algebra $R$, there is a unique map of $K$-algebras $\sigma_f: K[x] \rightarrow R$ such that $\sigma_f(x) = f$. In particular, for each polynomial $f(x) \in K[x]$ there is a $K$-algebra endomorphism $\sigma_f: K[x] \rightarrow K[x]$ such that $g(x)^{\sigma_f} := g(f(x))$ for all $g(x) \in K[x]$. Thus $K[x]$ carries a $K\{f\}$-composition algebra structure where $K\{f\} := K\otimes \ZZ\{f\}$ and $g(x) \circ f := g(f(x))$.

\begin{example}
We demonstrate these notions with a simple explicit example: If $g(x) \in K[x]$, then
\[
    g(x) \circ ([f^5] - [1])([f^3] - [1]) = (g(f^5(x)) - g(x))(g(f^3(x)) - g(x)).
\]
\end{example}

\begin{defn}
A polynomial $q(x) \in K[x]$ is \emph{$f$-stable} for $f(x) \in K[x]$ if $q(x)$ divides $q(f(x))$.
\end{defn}

If $q(x)$ is $f$-stable, then the endomorphism $\sigma_f: K[x] \rightarrow K[x]$ descends to an endomorphism of the quotient $K[x]/(q(x))$. Note that if $q(x)$ is squarefree, then $q(x)$ divides $q(f(x))$ if and only if $f$ maps the roots of $q(x)$ into themselves. More generally, let $v_\alpha(q(x))$ denote the valuation of $q(x)$ at $x - \alpha$, then $q(x)$ is $f$-stable if and only if $v_{f(\alpha)}(q(x)) \geq v_\alpha(q(x))$ for all roots $\alpha$ of $q$.

\begin{lemma}
\label{lem f stable}
Let $f(x) \in K[x]$ be a polynomial and let $m, n \in \NN$ such that $n \geq 1$. Then $f^{m+n}(x) - f^m(x)$ is $f$-stable
\end{lemma}

\begin{proof}
First suppose $f(x) = a_k x^k + a_{k-1}x^{k-1} + \ldots + a_1x + a_0 \in K(a_0, a_1, \ldots, a_k)[x]$ is the generic degree-$k$ polynomial over $K$. Lemma \ref{lem dynatomic squarefree} implies that $f^{m+n}(x) - f^m(x)$ is squarefree. The roots of $f^{m+n}(x) - f^m(x)$ are $f$-preperiodic hence closed under iteration by $f$. Therefore $f^{m+n}(x) - f^m(x)$ is $f$-stable. Stability is preserved under specialization.
\end{proof}

Lemma \ref{lem f stable} implies that $K[x]/(f^{m+n}(x) - f^m(x))$ inherits a $\ZZ\{f\}$-composition ring structure from $K[x]$. Furthermore, since $g(f^{m+n}(x)) \equiv g(f^m(x)) \bmod (f^{m+n}(x) - f^m(x))$ for all polynomials $g(x)$, the action of $\ZZ\{f\}$ factors through $\ZZ_{m,n}\{f\}$. This is summarized in the following lemma.

\begin{lemma}
\label{lem congruent operators}
Let $f(x) \in K[x]$, and let the composition ring $\ZZ\{f\}$ act on $K[x]$ by $g(x) \circ f := g(f(x))$. If $\alpha, \beta \in \ZZ\{f\}$ are elements such that $\alpha \equiv \beta \bmod \ZZ_{m,n}\{f\}$, then for all $g(x) \in K[x]$,
\[
    g(x) \circ \alpha \equiv g(x) \circ \beta \bmod (f^{m+n}(x) - f^m(x)).
\]
\end{lemma}

\section{Results}
\label{sec results}

With everything in place, we now prove the main result.
\begin{thm}
\label{thm main}
Let $K$ be a field, let $f(x) \in K[x]$ be a polynomial of degree at least 2, and let $c, d, m, n$ be integers with $c, m\geq 0$ and $d, n \geq 1$. Suppose that
\begin{enumerate}
    \item\label{item:main divisible} either $m > c$ or $n\nmid d$,
    \item the cocore of $d$ is at least $m - \max(c - 1,0)$, and
    \item $x^n - 1$ divides the $d$th necklace polynomial $M_d(x)$ in $\QQ[x]$.
\end{enumerate}
Then $\Phi_{f,m,n}(x)$ divides $\Phi_{f,c,d}(x) - 1$.

Alternatively, if $d > 1$, $c - 1 \geq m$, and $n = 1$, then $\Phi_{f,m,n}(x)$ divides $\Phi_{f,c,d}(x) - 1$.
\end{thm}

\begin{proof}[Proof of Theorem \ref{thm main}]
It suffices to prove the result for $f(x) \in K(a_0, a_1, \ldots, a_k)[x]$ the generic degree $k\geq 2$ polynomial over $K$. Suppose (1), (2), and (3) hold. We first prove the result assuming $c = 0$. Assumptions (2) and (3) imply that $\Phi_{f,d}^+ \equiv \Phi_{f,d}^- \bmod \ZZ_{m,n}\{f\}$ by Proposition \ref{prop necklace vanishing}.
If $\Phi_{f,d}^\pm(x) := x \circ \Phi_{f,d}^\pm$, then by Lemma \ref{lem congruent operators},
\begin{equation}
\label{eqn dynatomic congruence}
    \Phi_{f,d}^+(x) \equiv \Phi_{f,d}^-(x) \bmod (f^{m+n}(x) - f^m(x)).
\end{equation}
If $\alpha \in \overline{K(a_0, a_1, \ldots, a_k)}$ is a root of $\Phi_{f,m,n}(x)$, then Lemma \ref{lem dynamic factor} and \eqref{eqn dynatomic congruence} imply that
\[
    \Phi_{f,d}^+(\alpha) = \Phi_{f,d}^-(\alpha).
\]
If $m > 0$ or $n\nmid d$, then $f^e(\alpha) - \alpha \neq 0$ for any $e \mid d$ by Lemma \ref{lem dynatomic squarefree} and Lemma \ref{lem dynamic factor}; hence $\Phi_{f,d}^\pm(\alpha) \neq 0$.
Observe that
\begin{align*}
    \frac{\Phi_{f,d}^+(x)}{\Phi_{f,d}^-(x)}
    &= x \circ ([f] - [1])^{\varphi_d^+ - \varphi_d^-}\\
    &= x \circ ([f] - [1])^{\sum_{e\mid d}\mu(e)[d/e]}\\
    &= \prod_{e\mid d}(f^{d/e}(x) - x)^{\mu(e)}\\
    &= \Phi_{f,d}(x).
\end{align*}
Thus $\Phi_{f,d}(\alpha) = 1$. Since this holds for all roots $\alpha$ and $\Phi_{f,m,n}(x)$ is squarefree by Lemma \ref{lem dynatomic squarefree}, we conclude that $\Phi_{f,m,n}(x)$ divides $\Phi_{f,d}(x) - 1$.

Next suppose $c > 0$ and that the cocore of $d$ is at least $m - c + 1$. The above argument implies that $\Phi_{f,m-c + 1,n}(x)$ and $\Phi_{f,m-c,n}(x)$ divide $\Phi_{f,d}(x) - 1$. If $\alpha$ is a root of $\Phi_{f,m,n}(x)$, then $f^{c-i}(\alpha)$ is a root of $\Phi_{f,m-c +i,n}(x)$.
Hence $\Phi_{f,d}(f^c(\alpha)) = \Phi_{f,d}(f^{c-1}(\alpha)) = 1$ and
\[
    \Phi_{f,c,d}(\alpha) = \frac{\Phi_{f,d}(f^c(\alpha))}{\Phi_{f,d}(f^{c-1}(\alpha))} = \frac{1}{1} = 1.
\]
Thus $\Phi_{f,m,n}(x)$ divides $\Phi_{f,c,d}(x) - 1$ by Lemma \ref{lem dynatomic squarefree}.

Finally assume that $d > 1$, $c -1 \geq m$, and $n = 1$. If $\alpha$ is a root of $\Phi_{m,1}(x)$, then $c - 1 \geq m$ implies that $\beta := f^c(\alpha) = f^{c - 1}(\alpha)$. Furthermore, since $d > 1$ and $f(x)$ is generic, $\Phi_{f,d}(\beta) \neq 0$. Hence
\[
\Phi_{f,c,d}(\alpha) = \frac{\Phi_{f,d}(f^c(\alpha))}{\Phi_{f,d}(f^{c-1}(\alpha))} = \frac{\Phi_{f,d}(\beta)}{\Phi_{f,d}(\beta)} = 1.
\]
This identity holds for all $\alpha$ and $\Phi_{f,m,1}(x)$ is squarefree by Lemma \ref{lem dynatomic squarefree}, therefore $\Phi_{f,m,n}(x)$ divides $\Phi_{f,c,d}(x) - 1$.
\end{proof}

\begin{example}\label{ex necessary condition}
We show that condition \eqref{item:main divisible} from Theorem~\ref{thm main} is {\it generically} necessary, in the sense that if $f(x) = a_kx^k + \cdots + a_1x + a_0 \in K(a_0,a_1\ldots,a_k)[x]$ is the generic polynomial of degree $k$, and if $n,d \ge 1$ are integers satisfying $n \mid d$, then $\Phi_{f,0,n}(x) = \Phi_{f,n}(x)$ does not divide $\Phi_{f,d}(x) - 1$. If $n = d$, this is immediate so we assume that $n < d$.

Consider the polynomial $f(x) := x^k + a \in K(a)[x]$, where $a$ is an indeterminate. If the characteristic of $K$ does not divide $k$, then Theorem 2.2, Corollary 3.3, and Proposition 3.4 of \cite{MV} combine to show that the resultant $\Res(\Phi_{f,n}(x), \Phi_{f,d}(x))$ with respect to $x$ is a nonconstant polynomial in $K[a]$. Thus there exists $a_0 \in \overline{K}$ such that, for the polynomial $f_0(x) := x^k + a_0$, the dynatomic polynomials $\Phi_{f_0,n}(x)$ and $\Phi_{f_0,d}(x)$ have a common root $x_0$. (Over $\CC$, these values of $c_0$ are roots of hyperbolic components of the degree-$k$ multibrot set.) It follows that $\Phi_{f_0,n}(x)$ cannot divide $\Phi_{f_0,d}(x) - 1$, therefore this divisibility relation cannot hold generically.

Next suppose that the characteristic of $K$ divides $k$. Let $\zeta \in \overline{K}$ be a root of $\Phi_{d}(x)$, let $f(x) := x^k + \zeta x$, and let $\alpha$ be any root of $\Phi_{f,n}(x)$. Since $f'(x) = \zeta$, the period-$n$ multiplier of $\alpha$ is $(f^n)'(\alpha) = \zeta^n$, a root of $\Phi_{d/n}(x)$. It then follows from \cite[Thm. 2.2]{MV} that $\Res(\Phi_{f,n}(x), \Phi_{f,d}(x)) = 0$. Therefore $\Phi_{f,n}(x)$ and $\Phi_{f,d}(x)$ have a common root, whence $\Phi_{f,n}(x)$ does not generically divide $\Phi_{f,d}(x) - 1$.

\end{example}

\begin{example}
Condition (1) in Theorem \ref{thm main} is sufficient to guarantee that $\Phi_{f,d}^\pm(\alpha) \neq 0$ for any root $\alpha$ of $\Phi_{f,m,n}(x)$. If (1) fails to hold, deciding whether or not $\Phi_{f,m,n}(x)$ divides $\Phi_{f,d}(x) - 1$ is more subtle. 

Consider the quadratic polynomial family $f_a(x) = x^2 + a$. One may verify computationally that $\Phi_{f_a,6}(x) - 1$ factors over the function field $\QQ(a)$ as
\[
    \Phi_{f_a,6}(x) - 1 = h_a(x)\Phi_{f_a,1,2}(x)\Phi_{f_a,1,1}(x)
\]
where $h_a(x)$ is a degree 50 irreducible non-dynatomic polynomial with coefficients in $\QQ(a)$. The cocore of $d = 6$ is $1$ and 
\[
    M_6(x) = \tfrac{1}{6}(x^6 - x^3 - x^2 + x) = \tfrac{1}{6}(x^4 + x^2 - x)(x^2 - 1),
\]
hence conditions (2) and (3) of Theorem \ref{thm main} hold for $m = 0$ and $n = 2$, and yet $\Phi_{f_a,2}(x)$ does not generically divide $\Phi_{f_a,6}(x) - 1$. On the other hand, if $a = -1$ or $a = -5/4$, then one may check that $\Phi_{f_a,0,2}(x)$ does divide $\Phi_{f,6}(x) - 1$.
\end{example}

If $m = 0$ and $n = 1$, then condition (1) of Theorem \ref{thm main} is never satisfied. However, the following Proposition shows that in certain cases the conclusion of Theorem \ref{thm main} still holds.

\begin{prop}
\label{prop dynatomic cyclotomic}
Let $f(x) \in K[x]$ be a polynomial with fixed point $\alpha \in \overline{K}$,  let $\lambda := f'(\alpha)$ be the multiplier of $\alpha$, and let $d\geq 2$ be an integer, then
\[
    \Phi_{f,d}(\alpha) = \Phi_d(\lambda).
\]
Moreover, if $\lambda = 0$ or if 
\begin{enumerate}
    \item $\lambda$ is a primitive $n$th root of unity,
    \item $n$ is coprime to $d$, and
    \item $x^n - 1$ divides $M_d(x)$,
\end{enumerate}
then 
$
  \Phi_{f,d}(\alpha) = 1.
$
\end{prop}

Note that $\Phi_{f,d}(\alpha)$ is the $d$th dynatomic polynomial of $f(x)$ evaluated at a fixed point $\alpha$ and $\Phi_d(\lambda)$ is the $d$th cyclotomic polynomial evaluated at the multiplier $\lambda$ of $\alpha$.

\begin{proof}
Since $(f^k)'(\alpha) = \lambda^k$,
\[
    f^k(x) - x \equiv (\lambda^k - 1)(x - \alpha) \bmod (x - \alpha)^2.
\]
First suppose that $\lambda$ is not a $d$th root of unity. Then the $(x - \alpha)$-adic valuation of $f^e(x) - x$ is one for each $e\mid d$. Thus
\[
    \Phi_{f,d}(x) = \prod_{e\mid d}(f^e(x) - x)^{\mu(d/e)}
    = \prod_{e\mid d}\left(\frac{f^e(x) - x}{x - \alpha}\right)^{\mu(d/e)},
\]
where the second equality uses the fact that $\sum_{e\mid d}\mu(d/e) = 0$ for any $d\geq 2$. Evaluating at $x = \alpha$ gives
\[
    \Phi_{f,d}(\alpha) = \prod_{e\mid d}(\lambda^e - 1)^{\mu(d/e)} = \Phi_d(\lambda).
\]
Fix a degree $k$ and consider the affine algebraic variety
\[
    V_k := \{(f,\alpha) : \deg(f) \leq k \text{ and }\alpha \text{ is a fixed point of }f\}.
\]
The identity $\Phi_{f,d}(\alpha) = \Phi_d(f'(\alpha))$ holds on the Zariski open subset of all pairs $(f,\alpha)$ for which the multiplier $\lambda = f'(\alpha)$ is not a $d$th root of unity, hence it must hold on all of $V_k$.

If $\lambda = 0$, then $d \ge 2$ implies that $\Phi_d(0) = 1$, hence $\Phi_{f,d}(\alpha) = 1$.
Our assumption that $n$ is coprime to $d$ and that $x^n - 1$ divides $M_d(x)$ implies, by Theorem \ref{thm hyperplane}, that $\h{\U}_n \subseteq \bigcup_{p\mid d}\H_p$. Therefore $1 = \Phi_d(\lambda) = \Phi_{f,d}(\alpha)$ by \cite[Thm. 1.1]{hyde_cyclo}.
\end{proof}

\begin{remark}
The identity proved in Proposition \ref{prop dynatomic cyclotomic} is implicit in the proof of Theorem 2.2 of Morton and Vivaldi \cite{MV}; see the paragraph starting with display line (2.3). Hyde \cite[Thm. 1.8(2)]{hyde_cyclo} characterizes the pairs $(n,d)$ for which $d\nmid n$ and $\Phi_d(\zeta_n) = 1$. Using this characterization and Proposition \ref{prop dynatomic cyclotomic} one may construct special dynamical unit relations from fixed points $\alpha$ with $\lambda = \zeta_n$ which do not hold universally.

For example, one may check that $\Phi_{231}(\zeta_{12}) = 1$ and $x^{12} - 1$ does not divide $M_{231}(x)$. It is not generally the case that $\Phi_{f,231}(\alpha) = 1$ for fixed points $\alpha$, but this identity does hold if the multiplier of $\alpha$ is a primitive 12th root of unity (e.g. $f(x) = x^2 + \zeta_{12}x$ with $\alpha = 0$.)
\end{remark}

\subsection{Dynamical necklace polynomials}
The composition ring $\ZZ\{f\}$ also carries an additive $\Psi$-module structure where the natural action $[m]\cdot [f] := [f^m]$ is extended linearly. With respect to this structure we may define \emph{dynamical necklace polynomials} $M_{f,d}(x)$ analogous to the necklace polynomials $M_d(x)$,
\[
    M_{f,d}(x) := \frac{1}{d}\sum_{e\mid d}\mu(e) f^{d/e}(x) = (x/d) \circ \varphi_d [f].
\]
We are unaware of any natural interpretation, dynamical or otherwise, of the dynamical necklace polynomials $M_{f,d}(x)$. Nevertheless, the methods developed in the previous sections allow us to easily prove the following analog of Theorem \ref{thm main}.

\begin{prop}
Let $K$ be a field and let $f(x) \in K[x]$ be a polynomial. If
\begin{enumerate}
    \item the cocore of $d$ is at least $m$, and
    \item $x^n - 1$ divides the $d$th necklace polynomial $M_d(x)$ in $\QQ[x]$,
\end{enumerate}
then $f^{m+n}(x) - f^m(x)$ divides $M_{f,d}(x)$.
\end{prop}

\begin{proof}
Proposition \ref{prop necklace vanishing} and assumptions (2), (3) imply that $\varphi_d = 0 \bmod \ZZ\Psi_{m,n}$. Thus $\varphi_d \alpha \equiv 0 \bmod \ZZ_{m,n}\{f\}$ for any $\alpha \in \ZZ\{f\}$ by an additive version of Lemma \ref{lem compatible quotients}. Hence by Lemma \ref{lem congruent operators},
\[
    M_{f,d}(x) = (x/d) \circ \varphi_d[f] \equiv (x/d) \circ 0 \bmod (f^{m+n}(x) - f^m(x)).
\]
Note that for $r \in R$ an element of any composition algebra,
\[
    r \circ 0 = r \circ (0 + 0) = (r \circ 0) + (r \circ 0),
\]
hence $r \circ 0 = 0$. Thus
\[
    M_{f,d}(x) \equiv 0 \bmod (f^{m+n}(x) - f^m(x)),
\]
which is to say that $f^{m+n}(x) - f^m(x)$ divides $M_{f,d}(x)$.
\end{proof}

\subsection{Dynamical units}
\label{sec:dynamical units}
Theorem \ref{thm main} has implications for the construction of dynamical units.
Inspired by the theory of cyclotomic and elliptic units, Narkiewicz \cite{nark} and later Morton and Silverman \cite{MS} initiated the study of \emph{dynamical units}: algebraic units constructed in one of several closely related ways from differences of preperiodic points of a rational map of the projective line. 
The inspiration comes from the fact that, in the dictionary between dynamical height and the usual Weil height on the torus $\GG_m(\overline{\QQ})$, the preperiodic points play the same role as that of roots of unity, so the fields generated by these points are naturally thought of as \emph{dynatomic fields} in analogy with the classical theory of cyclotomic fields. 
We refer the reader to \cite[Section 3.11]{ADS} for further background on dynamical units.

Some families of dynamical units are known. Let $K$ be a number field with ring of integers $\OO_K$. Narkiewicz (\cite{nark}, cf. \cite[Thm. 6.3(a)]{MS}) proved that if $f\in \OO_K[x]$ is a monic polynomial of degree at least $2$, $\alpha\in\overline{K}$ is a root of $\Phi_{f,n}(x)$ for some $n\geq 2$, and $i,j\geq 0$ are integers such that $\gcd(i-j,n)=1$, then 
\[
 \frac{f^i(\alpha) - f^j(\alpha)}{f(\alpha)-\alpha}\in \OO_K^\times
\]
 is a dynamical unit. If $\zeta=\zeta_{p^m}$ denotes a primitive prime power order root of unity, then the reader will note the similarity to cyclotomic units in the maximal totally real subfield $\QQ(\zeta)^{\mathrm{tr}}$ of $\QQ(\zeta)$ given by
 \[
 \zeta^{(1-a)/2} \frac{1-\zeta^a}{1-\zeta},\quad\text{where}\quad 1<a<p^m/2\text{ and }\gcd(a,p)=1.
 \]
 It is known that units of this form, together with $-1$, generate the unit group of $\QQ(\zeta)^{\mathrm{tr}}$, and that this group has finite index in the unit group of $\QQ(\zeta)$.
 
 Morton and Silverman proved in \cite[Thm. 6.3(b)]{MS} (see also \cite[Prop. 7.4]{MS} for a formulation which is closer to our result) that if $f(x) \in \OO_K[x]$ is monic of degree at least $2$, and $\alpha,\beta\in \overline{K}$ are points of strict period $m$ and $n$ respectively, where $m,n\in\NN$ satisfy $m\nmid n$ and $n\nmid m$, then in fact
\[
 \alpha-\beta\in \OO_K^\times
\]
is a dynamical unit. 
Under the same assumptions on $f(x)$, Benedetto proved that if $m\geq 1$ and $\alpha$ is a root of $\Phi_{f,m,n}(x)$ and $\beta$ is a root of $\Phi_{f,d}(x)$ for some $n,d\geq 1$, then again, $\alpha-\beta\in \OO_K^\times$ (see \cite[Thm. 3]{benedetto}). 
Benedetto's result has interesting implications. For example, if $\{\alpha_1,\ldots,\alpha_n\}$ is an $n$-cycle for $f(x) = x^2+c$, that is, if $f(\alpha_1)=\alpha_2$, $f(\alpha_2)=\alpha_3$, \ldots, $f(\alpha_n)=\alpha_1$, then Benedetto shows \cite[Theorem 1]{benedetto} that 
 \[
 \prod_{i=1}^n (f(\alpha_i) + \alpha_i ) = 1,
 \]
 and in particular, that $f(\alpha)+ \alpha$ is a dynamical unit. This result is particularly remarkable as, from a dynamical perspective, one would not expect the sum of points to be related to the dynamics of a quadratic map. For a more recent result involving quadratic forms and dynamical units for rational maps, we also refer the reader to the work of Panraksa and Washington \cite{PW}.
  
Theorem \ref{thm main} allows us to deduce similar results about dynamical units, extending the results of Morton-Silverman and Benedetto. 
Note that if $f\in \OO_K[x]$ is monic, then 
$\Phi_{f,m,n}(x)\in \OO_K[x]$ is monic as well, and so our preperiodic points are algebraic integers. It follows that if $\Phi_{f,m,n}(x)$ divides $\Phi_{f,d}(x) - 1$, then for each root $\alpha \in \overline{K}$ of $\Phi_{f,m,n}(x)$,
\begin{equation}
\label{eqn unit relation body}
    1 = \Phi_{f,d}(\alpha) = \prod_{\beta}(\alpha - \beta),
\end{equation}
where the product ranges over all the roots $\beta$ of $\Phi_{f,d}(x)$ with multiplicity. 
Since $\alpha, \beta$ are algebraic integers, \eqref{eqn unit relation body} implies that the differences $\alpha - \beta$ are dynamical units.
If the conditions of Theorem \ref{thm main} are satisfied for $m, n , d$, then \eqref{eqn unit relation body} holds for all $f(x)$ with degree at least 2. We view these as \emph{universal relations} for dynamical units. 
In the case where the conditions of Theorem \ref{thm main} are met with $m=0$ and $n\nmid d$, we recover the result of Morton and Silverman quoted above. However, our result also applies in cases where the results of Morton and Silverman, and those of Benedetto, do not apply.

\begin{example}
If $(m,n,c,d)=(1,2,1,3)$, then the conditions of Theorem \ref{thm main} hold. Suppose that $K$ is a number field, $f(x)\in \OO_K[x]$ is a monic polynomial of degree at least $2$. If $\alpha, \beta \in \overline{K}$ are roots of $\Phi_{f,1,2}(x)$ and $\Phi_{f,1,3}(x)$, respectively, then $\alpha - \beta$ is a dynamical unit.
This class of dynamical units is new; the results of Morton-Silverman and Benedetto for differences of preperiodic points both required at least one of the points to be purely periodic, while both points here are \emph{strictly} preperiodic.
\end{example}

Morton and Silverman \cite[Prop. 7.4(b)]{MS} prove that if all the prime factors of $d > 1$ are congruent to $1 \bmod n$, then $\Phi_{f,n}(x)$ divides $\Phi_{f,d}(x) - 1$. This is a special case of Corollary \ref{cor:MS}.

\begin{cor}
\label{cor:MS}
Let $d > 1$ and $n \geq 1$ be integers such that $n \nmid d$ and suppose that $d$ is divisible by some prime $p \equiv 1 \bmod n$. Then $\Phi_{f,n}(x)$ divides $\Phi_{f,d}(x) - 1$.
\end{cor}

\begin{proof}
Recall that if $d = \prod_p p^{k_p}$ is the prime factorization of $d$, then $\varphi_d$ factors as
\[
    \varphi_d = \prod_{p} [p^{k_p-1}]([p] - [1]).
\]
Thus if $p \equiv 1 \bmod n$, then $\varphi_d \equiv 0 \brmod{n}$. The proof of Proposition \ref{prop necklace vanishing} shows that this is equivalent to $x^n - 1$ dividing $M_d(x)$. Conditions (1) and (2) of Theorem \ref{thm main} are trivially satisfied since $m, c = 0$, hence Theorem \ref{thm main} implies that $\Phi_{f,n}(x)$ divides $\Phi_{f,d}(x) - 1$.
\end{proof}

Note that if all the primes dividing $d$ are $1 \bmod n$, as is assumed in \cite[Prop. 7.4(b)]{MS}, then $d$ and $n$ are coprime, hence $n \nmid d$. Thus the Morton-Silverman result follows. In terms of hyperplanes covering the group of Dirichlet characters (see Section \ref{sec cyclo factors of neck}), the case $p \equiv 1 \bmod n$ for some prime $p \mid d$ corresponds to the situation where $\H_p = \widehat{\U}_n$ is the trivial hyperplane.

We can generalize the result of Benedetto in the following fashion:
\begin{prop}
\label{thm benedetto generalization}
 Suppose that $K$ is a number field with ring of integers $\OO_K$, $f(x) \in \OO_K[x]$ is monic of degree at least $2$, and $\beta\in \overline{K}$ is a root of $\Phi_{f,d}(x)$ for some $d\geq 2$. Then $\Phi_{f,1,1}(\beta)$ is a dynamical unit satisfying the relation
\begin{equation}
  \prod_{\substack{\beta\\ \Phi_{f,d}(\beta)=0}} \Phi_{f,1,1}(\beta) = 1
\end{equation}
where the product is taken over the roots of $\Phi_{f,d}$ with multiplicity.
\end{prop}

\begin{proof}
We begin by noting that if $m=n=1$ and $c=0$ and $d\geq 2$, then the indices meet the conditions \eqref{item:divisible}-\eqref{item:thm1-3} of Theorem \ref{thm main}: The first two conditions are obvious,  and the third follows from observing that
 \[
 M_d(1) = \frac1d \sum_{e\mid d} \mu(e) 1^{d/e} = 0
 \]
for all $d\geq 2$, so $(x-1)\mid M_d(x)$ in $\QQ[x]$. 
Thus Theorem \ref{thm intro main} guarantees that $\Phi_{f,1,1}(x)$ divides $\Phi_{f,d}(x) - 1$.
This means that if $\alpha$ is any root of $\Phi_{f,1,1}(x)$, then
\begin{equation}
\label{eqn:pre-benedetto}
  \Phi_{f,d}(\alpha) = \prod_{\substack{\beta\\ \Phi_{f,d}(\beta)=0}} (\alpha - \beta) = 1
\end{equation}
where the roots of $\beta$ of $\Phi_{f,d}$ are counted with multiplicity.
Taking the product of the identities \eqref{eqn:pre-benedetto} as $\alpha$ varies over the roots of $\Phi_{f,1,1}(x)$ with multiplicity gives
\[
  \Res(\Phi_{f,1,1}, \Phi_{f,d}) 
  = \prod_{\substack{\alpha\\ \Phi_{f,1,1}(\alpha)=0}}
  \prod_{\substack{\beta\\ \Phi_{f,d}(\beta)=0}} (\alpha - \beta) = 1.
\]
This resultant may also be expressed as,
\[
  \Res(\Phi_{f,1,1}, \Phi_{f,d}) 
  = \prod_{\substack{\beta\\ \Phi_{f,d}(\beta)=0}} \Phi_{f,1,1}(\beta) = 1
\]
 which gives us the desired result.
\end{proof}

 To see why this generalizes Benedetto's result, observe that when $f(x)=x^2+c$, one can check that
 \[
 \Phi_{f,1,1}(x) = \frac{\Phi_{f,1}(f(x))}{\Phi_{f,1}(x)} = \frac{f^2(x) - f(x)}{f(x) - x} = f(x)+x
 \]
 and we recover the result that $f(\alpha)+\alpha$ is a dynamical unit, although our multiplicative identity differs slightly from that of Benedetto, as it is a product over other points of formal period $d$ (that is, roots of $\Phi_{f,d}(x)$; for a review of the difference between formal and strict period, we refer the reader to \cite[\S 4.1]{ADS}), rather than the points directly in the cycle of $\alpha$. We can also easily find further examples of this sort:
\begin{cor}
 Suppose $K,\OO_K$ are as above and $f(x) = x^2 + b_1x + b_0\in \OO_K[x]$. Then for any $d\geq 2$, 
 \begin{equation}
 \prod_{\substack{\alpha\\ \Phi_{f,d}(\alpha)=0}} (f(\alpha) + \alpha + b_1) = 1,
\end{equation}
so $f(\alpha)+\alpha+b_1$ is a dynamical unit for any $\alpha\in\overline{K}$ of formal period $d\geq 2$. Likewise, if $f(x)=x^3 + 1\in \ZZ[x]$, then 
 \begin{equation}
 \prod_{\substack{\alpha\\ \Phi_{f,d}(\alpha)=0}} (1+ \alpha + \alpha^2 + 2\alpha^3 + \alpha^4 + \alpha^6) = 1.
\end{equation}
Thus if $\alpha\in \overline{K}$ is of formal period $d\geq 2$, then $1+ \alpha + \alpha^2 + 2\alpha^3 + \alpha^4 + \alpha^6$ is a dynamical unit.
\end{cor}

\subsection{Cyclotomic factors of necklace polynomials}
\label{sec cyclo factors of neck}
As discussed in the introduction, the most subtle condition in Theorem \ref{thm main} is $x^n - 1$ dividing $M_d(x)$. Theorem \ref{thm hyperplane} gives an alternative characterization of this divisibility in terms of hyperplane arrangements in finite abelian groups.

\begin{defn}
For $n\geq 1$, let $\U_n := (\ZZ/(n))^\times$ denote the multiplicative group of units modulo $n$ and let $\h{\U}_n := \mathrm{Hom}(\U_n, \CC^\times)$ denote the group of Dirichlet characters of modulus $n$. If $q \in \U_n$, then the \emph{hyperplane} $\H_q \subseteq \h{\U}_n$ is the set
\[
    \H_q := \{\chi \in \h{\U}_n: \chi(q) = 1\}.
\]
\end{defn}

\begin{thm}
\label{thm hyperplane}
Let $d, n \geq 1$. Then $x^n - 1$ divides $M_d(x)$ if and only if
\[
    \h{\U}_n \subseteq \bigcup_{\substack{p\mid d\\p\nmid n}}\H_p.
\]
\end{thm}

\begin{proof}
As we argued in the proof of Proposition \ref{prop necklace vanishing}, Lemma \ref{lem free module} implies that $\varphi_d \equiv 0 \brmod{n}$ if and only if $\varphi_d x/d = M_d(x)$ is divisible by $x^n - 1$. Let $\tilde{d}$ be the largest factor of $d$ coprime to $n$. The group ring $\QQ[\U_n]$ naturally embeds into $\QQ\Psi_{0,n}$ as the $\QQ$-span of $[q]$ for $q \in \U_n$, and $\varphi_{\tilde{d}} \in \QQ[\U_n] \subseteq \QQ\Psi_{0,n}$. Observe that
\[
    \varphi_{\tilde{d}} = \sum_{e\mid \tilde{d}}\mu(e)[\tilde{d}/e] = [\tilde{d}]\prod_{p\mid \tilde{d}}(1 - [p]^{-1}) \in \QQ[\U_n],
\]
where the product is taken over all primes $p$ dividing $\tilde{d}$. Recall that each character $\chi \in \h{\U}_n$ extends to a ring homomorphism $\chi: \QQ[\U_n] \rightarrow \CC$. Thus if $\chi \in \h{\U}_n$,  then
\[
    \chi(\varphi_{\tilde{d}}) = \chi(\tilde{d})\prod_{p\mid \tilde{d}}(1 - \overline{\chi(p)}).
\]
If $\chi_i \in \h{\U}_n$ for $1 \leq i \leq \varphi(n)$ are the distinct characters of $\U_n$, then the map
\[
    \alpha \in \QQ[\U_n]  \longmapsto (\chi_1(\alpha),  \chi_2(\alpha), \ldots,  \chi_{\varphi(n)}(\alpha)) \in \CC^{\varphi(n)}
\]
is an embedding of rings. Hence $\alpha = 0$ in $\QQ[\U_n]$ if and only if $\chi_i(\alpha) = 0$ for all $\chi_i$. Thus $\varphi_{\tilde{d}} = 0$ in $\QQ[\U_n]$ if and only if for each $\chi \in \h{\U}_n$ there is some prime $p \mid \tilde{d}$ such that $\chi(p) = 1$. This is equivalent to $\h{\U}_n \subseteq \bigcup_{p\mid \tilde{d}} \H_p = \bigcup_{p\mid d, p\nmid n} \H_p$.

Hence if $\h{\U}_n \subseteq \bigcup_{p\mid \tilde{d}} \H_p$, then
\[
    dM_d(x) = \varphi_d x = (\varphi_{\tilde{d}}\varphi_{d/\tilde{d}}) x = \varphi_{\tilde{d}}\cdot  (\varphi_{d/\tilde{d}} x) \equiv 0 \cdot (\varphi_{d/\tilde{d}} x) \equiv 0 \bmod x^n - 1.
\]

Conversely, suppose that $x^n - 1$ divides $M_d(x)$. Let $U \subseteq \QQ[x]/(x^n - 1)$ denote the $\QQ$-subspace spanned by $x^j$ with $j$ coprime to $n$, and let $S_d(x) := dM_d(x)$. Observe that
\[
    S_d(x) = \varphi_d x = \varphi_{d/\tilde{d}}S_{\tilde{d}}(x)= \sum_{e \mid d/\tilde{d}}\mu(d/\tilde{d}e)S_{\tilde{d}}(x^e). 
\]
Since $\tilde{d}$ is the largest factor of $d$ coprime to $n$, it follows that each $e > 1$ dividing $d/\tilde{d}$ shares a nontrivial common factor with $n$. Hence the $U$-component of $S_d(x)$ is $\pm S_{\tilde{d}}(x)$. Hence if $S_d(x) \equiv 0 \bmod x^n - 1$, then it must be the case that $S_{\tilde{d}}(x) \equiv 0 \bmod x^n - 1$. As argued above, this is equivalent to $\h{\U}_n \subseteq \bigcup_{p\mid \tilde{d}}\H_p$.
\end{proof}

\begin{remark}
Theorem \ref{thm hyperplane} is closely related to \cite[Thm. 1.13]{hyde_cyclo} but with a slightly different scope. Neither result directly implies the other. 
\end{remark}

\begin{example}
\label{ex m = 65}
The following example is adapted from \cite[Ex. 2.8]{hyde_cyclo}. Let $d = 440512358437 = 47^2 \cdot 73 \cdot 79 \cdot 151 \cdot 229$ and let $n = 65$. The group $\h{\U}_{65} \cong (\ZZ/(65))^\times$ decomposes as $\h{\U}_{65} \cong \ZZ/(4)^2 \times \ZZ/(3)$. Note that each hyperplane $\H_p \subseteq \h{\U}_{65}$ is a subgroup, hence factors as $\H_p \cong \H_p^{(4)} \times \H_p^{(3)}$ with $\H_p^{(4)} \subseteq \ZZ/(4)^2$ and $\H_p^{(3)} \subseteq \ZZ/(3)$. In this case, each of the hyperplanes $\H_p$ with $p \mid d$ is trivial in the 3-torsion $\H_p^{(3)} = \ZZ/(3)$. Thus it suffices to consider the $4$-torsion $\H_p^{(4)}$ of each hyperplane $\H_p$.

Identifying the $4$-torsion of $\h{\U}_{65}$ with the additive group $\ZZ/(4)^2$, the group $\U_{65} := (\ZZ/(65))^\times$ of units modulo $65$ has a compatible isomorphism $\rho: \U_{65} \rightarrow \langle x, y : 4x = 4y = 0\rangle$ with the dual group of $\ZZ/(4)^2$. With respect to such an isomorphism, (the $4$-torsion of) each hyperplane $\H_p$ may be realized as the vanishing set of a homogeneous linear form, hence the hyperplane terminology.

The units $47$ and $151$ generate a $\ZZ/(4)^2$ subgroup of $\U_{65}$, so we may choose coordinates $\rho$ such that $x := \rho(47)$ and $y := \rho(151)$. Then the hyperplanes $\H_p$ may be visualized as lines in the ``plane'' $(\RR/4\ZZ)^2$. Each of the five distinct primes dividing $d$ corresponds to a different colored line in the diagram below. For example, since $229 \equiv 47^2 \cdot 151^{-1} \bmod 65$, the (4-torsion of the) hyperplane $\H_{229}$ is the solution set of $2x - y = 0$ in $\ZZ/(4)^2$.
Figure \ref{fig:65 body} shows the linear forms defining each line with respect to this choice of coordinates.

\begin{figure}[h]
    \centering
    \includegraphics[scale=.19]{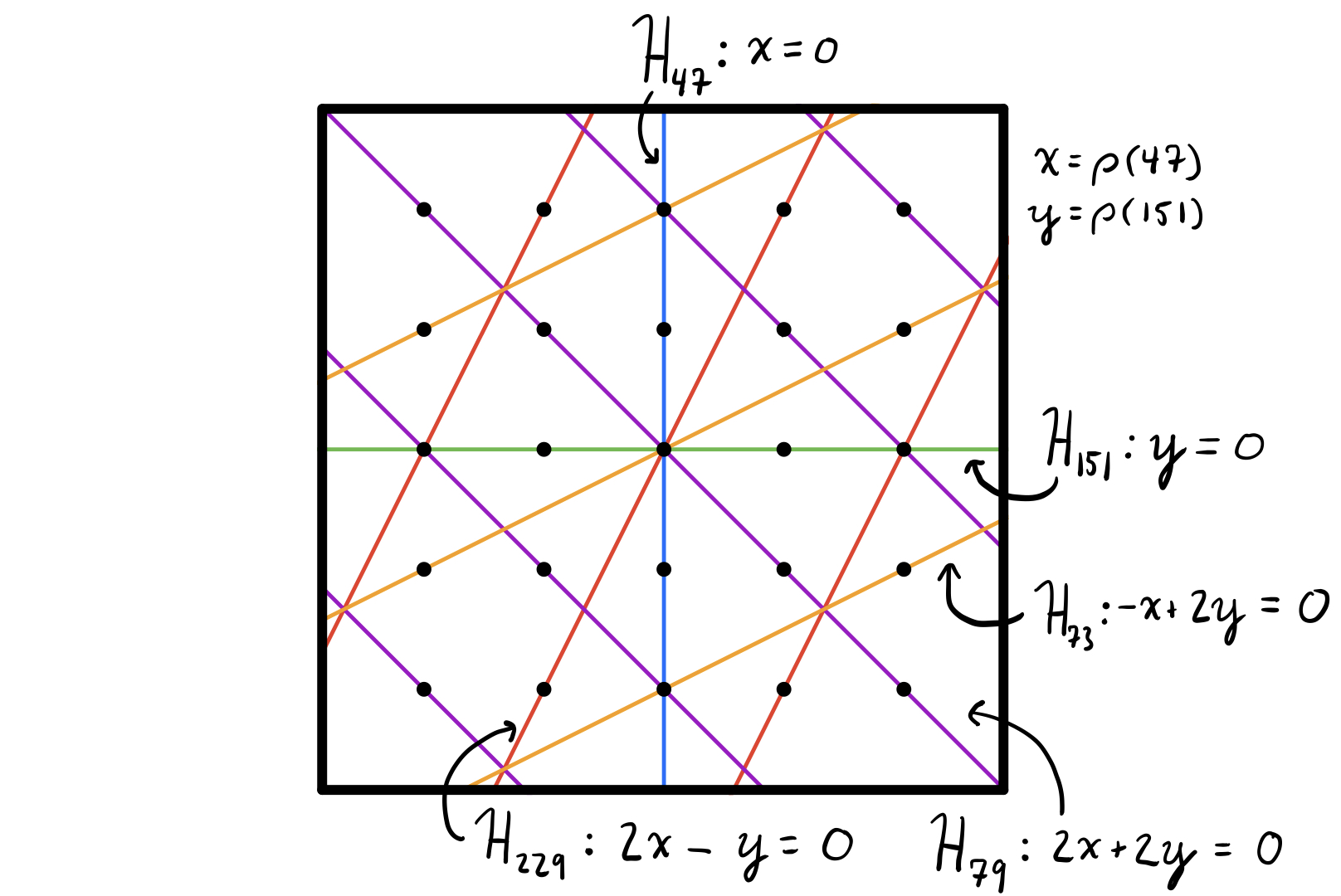}
    \caption{The lattice points in $(\RR/4\ZZ)^2$ may naturally be identified with $\ZZ/(4)^2$.}
    \label{fig:65 body}
\end{figure}

Since the five lines $\H_p$ with $p\mid d$ cover all of $\ZZ/(4)^2$, it follows that
$
    \h{\U}_{65} \subseteq \bigcup_{p\mid d}\H_p,
$
with $d = 440512358437$. The cocore of $d$ is 47. Hence Theorem \ref{thm intro main} implies that for any polynomial $f(x) \in K[x]$ with degree at least 2 and any $m \leq 47$,
\[
    \Phi_{f,m,65}(x) \text{ divides } \Phi_{f,440512358437}(x) - 1.
\]

By drawing other arrangements of lines covering $\ZZ/(4)^2$ and finding primes in the corresponding congruence classes modulo $65$ (which must exist by Dirichlet's theorem on primes in arithmetic progressions) we may construct several other nontrivial examples of $d$ for which $\h{\U}_{65} \subseteq \bigcup_{p\mid d}\H_p$. Three examples are given in Figure \ref{fig: arrangements}.

\begin{figure}
\centering
\includegraphics[scale=.19]{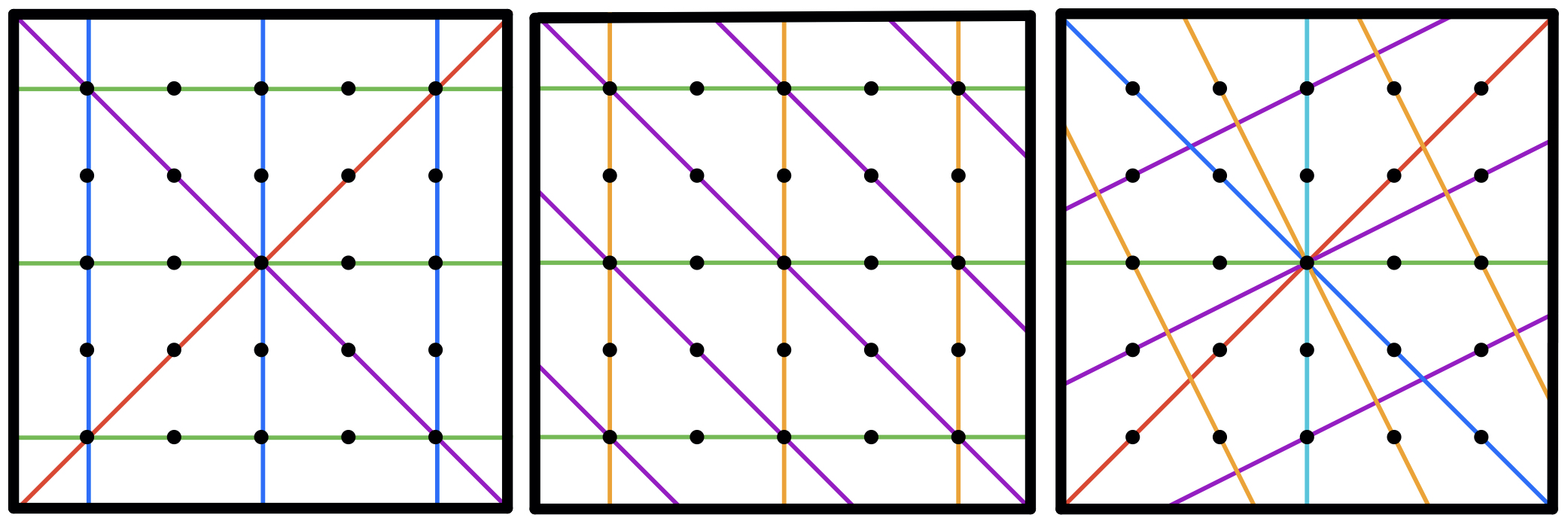}
\caption{}
\label{fig: arrangements}
\end{figure}

Values of $d$ corresponding to the three arrangements in Figure \ref{fig: arrangements} are, respectively,
\begin{align*}
    d_1 &= 157\cdot 181 \cdot 337 \cdot 389\\
    d_2 &= 79 \cdot 181 \cdot 389\\
    d_3 &= 47 \cdot 109 \cdot 151 \cdot 157 \cdot 317 \cdot 337.
\end{align*}
Each of these $d_i$ are squarefree and coprime to $65$, so it follows that 
\[
    \Phi_{f,m,65}(x) \text{ divides } \Phi_{f,d_i}(x) - 1
\]
for each $m = 0, 1$ and each $d_i$.
\end{example}


\begin{thebibliography}{99}
\bibitem{benedetto}
R. L. Benedetto, An elementary product identity in polynomial dynamics, \emph{Am. Math. Mon.}, \textbf{108}, no. 9 (2001), 860--864.

\bibitem{doyle/poonen}
J. R. Doyle and B. Poonen, Gonality of dynatomic curves and strong uniform boundedness of preperiodic points, \emph{Compos. Math.}, \textbf{156} (2020), 733--743.


\bibitem{fakhruddin}
N. Fakhruddin, The algebraic dynamics of generic endomorphisms of $\mathbb{P}^n$, \emph{Algebra Number Theory}, \textbf{8} (2014), 587--608.

\bibitem{hutz}
B. Hutz, Determination of all rational preperiodic points for morphisms of $\PP^N$, \emph{Math. Comp.}, \textbf{84}, no. 291 (2015), 289--308.

\bibitem{hyde_cyclo}
T. Hyde, Cyclotomic factors of necklace polynomials, arXiv:1811.08601, 2020.

\bibitem{morton:1996}
P. Morton, On certain algebraic curves related to polynomial maps, \emph{Compos. Math.}, \textbf{103}, no. 3 (1996), 319--350.

\bibitem{morton patel}
P. Morton, P. Patel, The Galois theory of periodic points of polynomial maps, \emph{Proc. London Math. Soc. (3)}, vol. 68, no. 2 (1994), 225--263.

\bibitem{MS}
P. Morton, J. H. Silverman, Periodic points, multiplicities, and dynamical units, \emph{J. reine angew. Math.}, \textbf{461} (1995), 81--122.

\bibitem{MV}
P. Morton, F. Vivaldi, Bifurcations and discriminants for polynomial maps, \emph{Nonlinearity}, \textbf{8} (1995), 571--584.

\bibitem{nark}
W. Narkiewicz, Polynomial cycles in algebraic number fields, \emph{Colloq. Math.}, \textbf{58}, no. 1 (1989), 151--155.

\bibitem{PW}
C. Panraksa, L. Washington, Arithmetic dynamics and dynamical units,
 \emph{East-West J. Math.}, \textbf{14}, no. 2 (2012), 201--207.

\bibitem{ADS}
J. H. Silverman, The arithmetic of dynamical systems, \emph{Spring Science \& Business Media}, \textbf{241} (2007).

\end{thebibliography}
\end{document}